\documentclass[11pt,reqno]{amsart}

\usepackage[margin=1in]{geometry}
\usepackage{amsmath,amssymb,amsthm,amsfonts,mathtools}
\usepackage{enumitem}
\usepackage{booktabs}
\usepackage{microtype}

\usepackage{thmtools}
\usepackage[colorlinks=true,linkcolor=blue,citecolor=magenta,urlcolor=blue]{hyperref}
\usepackage[capitalize,nameinlink]{cleveref}

\theoremstyle{plain}
\newtheorem{theorem}{Theorem}[section]
\newtheorem{lemma}[theorem]{Lemma}
\newtheorem{proposition}[theorem]{Proposition}

\newtheorem{conjecture}[theorem]{Conjecture}
\theoremstyle{definition}

\theoremstyle{remark}

\numberwithin{equation}{section}

\newcommand{\ZZ}{\mathbb{Z}}
\newcommand{\QQ}{\mathbb{Q}}
\newcommand{\RR}{\mathbb{R}}
\newcommand{\NN}{\mathbb{N}}
\newcommand{\CC}{\mathbb{C}}

\newcommand{\pre}{\mathrm{pre}}

\newcommand{\flip}{\mathrm{flip}}
\newcommand{\supp}{\mathrm{supp}}

\newcommand{\Res}{\mathrm{Res}}

\newcommand{\leftb}{\{\!\!\{}
\newcommand{\rightb}{\}\!\!\}}
\newcommand{\sms}[2]{#1^{(#2)}}
\newcommand{\Part}{\mathcal{P}}

\newcommand{\Singular}{\mathcal{M}}
\newcommand{\eqs}[1]{\overset{#1}{\sim}}

\begin{document}

\title[Elementary Symmetric Partitions and Multiset Recovery]{On the Injectivity of Elementary Symmetric Partitions and the Multiset Recovery Problem}

\author[Z. Sun]{Ziyao Sun} 
\address[Z. Sun]{College of Mathematics and Statistics, Chongqing University, Chongqing 401331, PR China}
\email{sunziyao79@gmail.com}


\begin{abstract}
	The elementary symmetric partition map $\pre_s$, introduced by Ballantine, Beck, and Merca, sends an integer partition to the summands in the evaluation of the $s$-th elementary symmetric polynomial at its parts. By encoding partition parts as prime-exponent valuation vectors, we connect $\pre_s$ to Leo Moser's additive Multiset Recovery Problem (1957) and prove that $\pre_s$ is unconditionally injective on partitions of length $n$ whenever $n$ lies outside the Moser root set $\mathcal{Z}_s$, with no size restrictions. Furthermore, under the equal-size constraint $|\lambda| = |\mu| = N$, we prove that $\pre_4$ is injective at the isolated singular length $n = 12$, and that every fiber of $\pre_3$ on $\Part_6(N)$ has cardinality at most $2$, completely excluding both triplets and quartets.
\end{abstract}

\maketitle

\section{Introduction}
\label{sec:intro}

A fundamental inverse problem at the interface of partition theory and symmetric functions is whether the evaluation of a symmetric polynomial on the parts of an integer partition uniquely determines the underlying partition itself. 

Throughout this paper, let $N$ and $n$ be positive integers. A \emph{partition} $\lambda = (\lambda_1, \dots, \lambda_n)$ of size $|\lambda| = \sum_{i=1}^n \lambda_i = N$ is a weakly decreasing sequence of positive integers. We denote the set of all partitions of size $N$ by $\Part(N)$, and the set of partitions of length $n$ by $\Part_n(N)$. We identify a partition with its multiset of parts $\leftb \lambda_1, \dots, \lambda_n \rightb$, using double-brace notation to distinguish multisets from sets.

For an integer $s \ge 1$, the $s$-th elementary symmetric polynomial in $n$ variables is defined by
\begin{equation*}
	e_s(x_1, \dots, x_n) = \sum_{1 \le i_1 < i_2 < \dots < i_s \le n} x_{i_1} x_{i_2} \cdots x_{i_s}.
\end{equation*}
The study of elementary symmetric partitions was initiated by Ballantine, Beck, and Merca \cite{ballantine2025partitions}. Specifically, they introduced the \emph{elementary symmetric partition map} $\pre_s$, which maps a partition $\lambda = (\lambda_1, \dots, \lambda_n)$ with $n \ge s$ to the partition formed by the $\binom{n}{s}$ summands in the expansion of $e_s(\lambda_1, \dots, \lambda_n)$:
\begin{equation*}
	\pre_s(\lambda) := \leftb \lambda_{i_1} \lambda_{i_2} \cdots \lambda_{i_s} : 1 \le i_1 < i_2 < \dots < i_s \le n \rightb.
\end{equation*}
Subsequently, together with Sagan \cite{ballantine2024elementary}, they extensively investigated the combinatorial properties of these partitions and explored their connections with various partition classes.

Ballantine, Beck, and Merca \cite{ballantine2025partitions} originally posed the following natural conjecture regarding the reconstructive power of this map, which was also highlighted in \cite{ballantine2024elementary}.

\begin{conjecture}[\cite{ballantine2025partitions}]
	\label{conj:bbm}
	For integers $s \ge 2$ and $N \ge 1$, the map $\pre_s$ is injective on the set of partitions of size $N$ with length $n \ge s$.
\end{conjecture}

Li \cite{li2026injectivity} proved the $s = 2$ case, in fact for weakly decreasing sequences of positive real numbers of fixed sum. However, Devnani and Eyyunni \cite{devnani2026elementary} disproved the boundary case $n = s$ by constructing infinite families of counterexamples, proposing in its place a refined conjecture.

\begin{conjecture}[\cite{devnani2026elementary}]
	\label{conj:de}
	For integers $s \ge 2$ and $N \ge 1$, the map $\pre_s$ is injective on the set of partitions of size $N$ with length strictly greater than $s$ ($n > s$).
\end{conjecture}

Very recently, \cref{conj:de} was answered in the negative by Hadelyn, Niergarth, Li, and Li \cite{hadelyn2026counterexamples}. Specifically, Hadelyn et al.\ proved that for every $s \ge 3$, there exist infinitely many pairs of distinct partitions $\lambda, \mu$ of length $2s$ with $$|\lambda|=|\mu|,\qquad \pre_s(\lambda)=\pre_s(\mu).$$ The case $s=3$ was independently obtained by Thomas and Tung \cite{thomas2026injectivitysymmetricpolynomialmaps}. Hadelyn et al.'s construction relied on an algebraic involution $\flip_s: \Part_{2s}(\RR) \to \Part_{2s}(\RR)$ defined by
\begin{equation*}
	\flip_s(\lambda_1, \dots, \lambda_{2s}) := \left( \prod_{i=1}^{2s} \lambda_i \right)^{1/s} \left( \frac{1}{\lambda_{2s}}, \dots, \frac{1}{\lambda_1} \right),
\end{equation*}
which preserves the image of $\pre_s$. This led them to propose the following bold conjecture:

\begin{conjecture}[\cite{hadelyn2026counterexamples}]
	\label{conj:hadelyn}
	Let $\lambda$ and $\mu$ be distinct partitions of length $n > s \ge 2$. Then $\pre_s(\lambda) = \pre_s(\mu)$ if and only if $n = 2s$ and $\mu = \flip_s(\lambda)$.
\end{conjecture}

\cref{conj:hadelyn} claims that for all lengths $n \notin \{s, 2s\}$, $\pre_s$ is injective without any size requirement, and that whenever non-injectivity occurs, the preimage fiber has cardinality at most $2$.

\subsection{The Additive Perspective: Multiset Recovery and Moser Polynomials}

A central thesis of the present paper is that the injectivity of $\pre_s$ is governed by an older, classical problem in additive number theory: the \emph{Multiset Recovery Problem} proposed by Leo Moser in 1957 \cite{moser1957problem}. In its general formulation, the problem asks whether a finite multiset of $n$ numbers $A = \leftb a_1, \dots, a_n \rightb$ in an abelian group can be uniquely reconstructed from the multiset of its $s$-sums:
\begin{equation*}
	\sms{A}{s} := \leftb a_{i_1} + a_{i_2} + \dots + a_{i_s} : 1 \le i_1 < i_2 < \dots < i_s \le n \rightb.
\end{equation*}
Two multisets $A$ and $B$ are called \emph{$s$-equivalent}, denoted $A \eqs{s} B$, if $\sms{A}{s} = \sms{B}{s}$. A pair $(s, n)$ with $n > s$ is called \emph{singular} if there exist distinct multisets $A \ne B$ of size $n$ such that $A \eqs{s} B$. We write $\Singular_s$ for the set of all integers $n > s$ for which $(s, n)$ is singular.

Selfridge and Straus \cite{selfridge1958determination} settled the case $s=2$, proving that $(2,n)$ is singular if and only if $n$ is a power of $2$. For general $s\ge2$, the reconstruction problem can be studied through the power sums of the $s$-fold sums. In particular, these power sums satisfy the following recursive expansion:
\begin{equation}
	\label{eq:intro_moser_recurrence}
	\sigma_k(\sms{A}{s}) = F_{s,k}(n) \sigma_k(A) + \mathcal{P}_k(\sigma_1(A), \dots, \sigma_{k-1}(A)),
\end{equation}
where $\mathcal{P}_k$ is an explicit polynomial, and $F_{s,k}(n)$ is the \emph{Moser polynomial}:
\begin{equation}
	\label{eq:intro_moser_poly}
	F_{s,k}(n) := \sum_{p=1}^s (-1)^{p-1}p^{k-1}\binom{n}{s-p}.
\end{equation}
This power-sum framework underlies the subsequent study of the exceptional cases by Gordon, Fraenkel, and Straus \cite{gordon1962determination}, Ewell \cite{ewell1968determination}, and Fomin and Izhboldin \cite{fomin1994sets}. Whenever $F_{s,k}(n) \ne 0$ for all $1 \le k \le n$, the power sums $\sigma_k(A)$ are determined recursively and uniquely from $\sms{A}{s}$, guaranteeing unique recovery. We define the \emph{Moser root set}:
\begin{equation*}
	\mathcal{Z}_s := \left\{ n \in \NN : n > s \text{ and } \exists k \in \{1, 2, \dots, n\} \text{ such that } F_{s,k}(n) = 0 \right\}.
\end{equation*}
By \eqref{eq:intro_moser_recurrence}, singularity strictly requires the vanishing of at least one Moser polynomial, establishing the containment $\Singular_s \subseteq \mathcal{Z}_s$. 

Decades of research have classified the exact root sets $\mathcal{Z}_s$ for small values of $s$:
\begin{itemize}[leftmargin=2em]
	\item For $s = 2$, $\mathcal{Z}_2 = \{2^m : m \ge 2\} = \{4, 8, 16, 32, \dots\}$ \cite{selfridge1958determination};
	\item For $s = 3$, $\mathcal{Z}_3 = \{6, 27, 486\}$ \cite{selfridge1958determination};
	\item For $s = 4$, $\mathcal{Z}_4 = \{8, 12\}$ \cite{selfridge1958determination};
	\item For $s = 5$, $\mathcal{Z}_5 = \{10\}$ \cite{amdeberhan1997injectivity}.
\end{itemize}
Furthermore, for every $s \ge 2$, $n = 2s$ is always singular: for any asymmetric multiset $A$ with mean $a$, its mirror multiset $\tilde{A} = 2a - A$ satisfies $A \eqs{s} \tilde{A}$.

\medskip
Here an immediate, sharp tension emerges between the two fields. Hadelyn et al.\ conjectured that non-injectivity for elementary symmetric partitions occurs \emph{if and only if} $n = 2s$. However, in the additive setting, $(s, n)$ is singular at other lengths, such as $n = 12$ for $s = 4$, and $n \in \{27, 486\}$ for $s = 3$ \cite{fomin1994sets,boman1996examples,isomurodov2017set}. If the size of the partitions is unconstrained, \cref{conj:hadelyn} fails immediately. This raises the central question of this paper: \emph{Does the equal-size constraint $|\lambda| = |\mu| = N$ eliminate non-injectivity at the non-mirror singular lengths, and does it enforce a universal fiber bound of $2$?}

\subsection{Main Results of This Paper}

By mapping each part $\lambda_i$ of a partition to its prime-exponent valuation vector $\mathbf{v}(\lambda_i) \in \ZZ_{\ge 0}^\infty$ under prime factorization, the multiplicative action of $\pre_s$ on $\Part_n(N)$ transforms isomorphically into the additive $s$-sum recovery of a multiset of vectors in $\ZZ_{\ge 0}^\infty$. 

Our first main result establishes the global domain on which $\pre_s$ is unconditionally injective.

\begin{theorem}
	\label{thm:main_unconditional}
	Let $s \ge 2$ and $n > s$. If $F_{s,k}(n) \ne 0$ for all $k \in \{1, 2, \dots, n\}$ (that is, $n \notin \mathcal{Z}_s$), then the elementary symmetric partition map $\pre_s$ is injective on partitions of length $n$, regardless of whether the partitions have the same size. In particular:
	\begin{enumerate}[label=\textup{(\roman*)}]
		\item For $s = 5$, $\pre_5$ is injective on $\Part_n$ for every $n \in \NN \setminus \{5, 10\}$.
		\item For $s = 4$, $\pre_4$ is injective on $\Part_n$ for every $n \in \NN \setminus \{4, 8, 12\}$.
		\item For $s = 3$, $\pre_3$ is injective on $\Part_n$ for every $n \in \NN \setminus \{3, 6, 27, 486\}$.
	\end{enumerate}
\end{theorem}

Our second main result resolves the exceptional singular length $(s, n) = (4, 12)$ under the equal-size constraint.

\begin{theorem}
	\label{thm:main_s4}
	Let $\lambda, \mu \in \Part(N)$ be partitions of the same size $N$ having length $n = 12$. If $\pre_4(\lambda) = \pre_4(\mu)$, then $\lambda = \mu$. Consequently, on partitions of equal size, the only length $n > 4$ at which $\pre_4$ fails to be injective is the mirror length $n = 8$.
\end{theorem}

The proof of \cref{thm:main_s4} proceeds by two structural steps. First, we establish a simultaneous collinearity lemma: by projecting multidimensional vector multisets into $\ZZ$ via $L_M(\mathbf{x}) = \sum x_k M^{k-1}$ and applying the Pigeonhole Principle across infinitely many $M$, forces the exponent vectors to lie on a common line, generated by a single rational base $r \in \QQ_{>0}$. Second, the equal-size constraint $|\lambda| = |\mu|$ reduces to the univariate polynomial $P(r) = \sum_{a \in A_0} r^a - \sum_{b \in B_0} r^b$, whose leading and constant coefficients are both $-1$; the Rational Root Theorem then forces the only positive rational root to be $r = 1$, which proves $\lambda = \mu$.

Our third main result bounds the preimage fibers of $\pre_3$ at the mirror length $n = 2s = 6$. While pairs of counterexamples exist via $\flip_3$, we prove that the fiber size cannot exceed $2$.

\begin{theorem}
	\label{thm:main_s3_fibers}
	For every integer $N \ge 1$, the elementary symmetric partition map $\pre_3$ restricted to $\Part_6(N)$ has fibers of cardinality at most $2$. That is, there do not exist three distinct partitions $\lambda_1, \lambda_2, \lambda_3 \in \Part_6(N)$ of equal size $N$ such that
	\begin{equation*}
		\pre_3(\lambda_1) = \pre_3(\lambda_2) = \pre_3(\lambda_3).
	\end{equation*}
\end{theorem}

To prove \cref{thm:main_s3_fibers}, we analyze the two four-member equivalence classes (quartets) classified by Ewell \cite{ewell1968determination}. For the single-parameter family, we eliminate non-trivial solutions using polynomial greatest common divisors over $\QQ[r]$. For the two-parameter family, we prove a simultaneous coplanarity lemma and eliminate the resulting bivariate Laurent system using Cylindrical Algebraic Decomposition (CAD). 

Finally, in Section~\ref{sec:conjectures}, we analyze the remaining singular lengths $n \in \{27, 486\}$ for $s = 3$. We show that every known counterexample family in the literature collapses under the equal-size constraint, and we formulate refined conjectures on the injectivity of elementary symmetric partitions.

\subsection{Organization of the Paper}

The remainder of this paper is organized as follows. In \cref{sec:bridge}, we formalize the prime-exponent vector embedding and prove \cref{thm:main_unconditional}. In \cref{sec:s4_proof}, we prove \cref{thm:main_s4}, establishing the simultaneous collinearity reduction and the rational root elimination for $(s, n) = (4, 12)$. In \cref{sec:s3_quartets}, we analyze Ewell's quartet classification, prove the rigidity of triplets, and establish \cref{thm:main_s3_fibers}. In \cref{sec:conjectures}, we verify the historical counterexample families at $n \in \{27, 486\}$ and propose refined conjectures for equal-size partitions. In Appendix~\ref{sec:appendix_code}, we provide the complete symbolic verification scripts used in the proofs.

\section{The Prime-Exponent Embedding and Moser Polynomials}
\label{sec:bridge}

In this section, we establish the formal bridge between the multiplicative action of $\pre_s$ and the additive multiset recovery problem. Throughout, we consider partitions $\lambda = (\lambda_1, \dots, \lambda_n)$ and $\mu = (\mu_1, \dots, \mu_n)$ of fixed length $n > s \ge 2$, with no size restrictions.

\subsection{The Monoid Isomorphism}

Let $\mathbb{P} = \{p_1, p_2, p_3, \dots\} = \{2, 3, 5, \dots\}$ be the sequence of prime numbers in increasing order. By the Fundamental Theorem of Arithmetic, every positive integer factors uniquely into prime powers. 

Consider the free abelian monoid $(\ZZ_{\ge 0}^\infty, +)$ of sequences of non-negative integers with finite support. The prime valuation map is defined by
\begin{equation*}
	\mathbf{v}: \NN \to \ZZ_{\ge 0}^\infty, \qquad x = \prod_{j=1}^\infty p_j^{v_j(x)} \longmapsto \mathbf{v}(x) = (v_1(x), v_2(x), \dots).
\end{equation*}
Because prime factorization is unique, $\mathbf{v}$ is an isomorphism of abelian monoids:
\begin{equation*}
	\mathbf{v}(x y) = \mathbf{v}(x) + \mathbf{v}(y) \qquad (\forall x, y \in \NN).
\end{equation*}

For any partition $\lambda = (\lambda_1, \dots, \lambda_n)$, we define its exponent vector multiset by
\begin{equation*}
	V_\lambda := \leftb \mathbf{v}(\lambda_1), \mathbf{v}(\lambda_2), \dots, \mathbf{v}(\lambda_n) \rightb \subset \ZZ_{\ge 0}^\infty.
\end{equation*}
Because $\mathbf{v}$ is bijective, $\lambda$ is uniquely determined by the multiset $V_\lambda$, which immediately yields the following characterization.

\begin{lemma}
	\label{lem:multiplicative_to_additive}
	Let $s \ge 2$ and $n \ge s$. For any two partitions $\lambda, \mu$ of length $n$,
	\begin{equation*}
		\pre_s(\lambda) = \pre_s(\mu) \iff \sms{V_\lambda}{s} = \sms{V_\mu}{s},
	\end{equation*}
	where $\sms{V_\lambda}{s}$ is the multiset of $s$-sums of $V_\lambda$ in $\ZZ_{\ge 0}^\infty$.
\end{lemma}

\subsection{Moser Polynomials and the Recovery Criterion}

The monoid $\ZZ_{\ge 0}^\infty$ embeds into the free abelian group $\ZZ^\infty$, which is torsion-free. A classical theorem of Gordon, Fraenkel, and Straus ensures that multiset recovery over torsion-free abelian groups reduces to the integers.

\begin{lemma}[\cite{gordon1962determination}]
	\label{lem:transfer}
	Let $s \ge 2$ and $n > s$. The multiset recovery problem of $s$-sums over any torsion-free abelian group is equivalent to the recovery problem over $\ZZ$. 
\end{lemma}

Let $A = \leftb a_1, \dots, a_n \rightb \subset \ZZ$ be a multiset of integers. For each integer $k \ge 1$, we define the power sums of $A$ and of its $s$-sums $\sms{A}{s}$ by
\begin{equation*}
	\sigma_k(A) := \sum_{i=1}^n a_i^k, \qquad \sigma_k(\sms{A}{s}) := \sum_{1 \le i_1 < \dots < i_s \le n} (a_{i_1} + \dots + a_{i_s})^k.
\end{equation*}

\begin{lemma}[\cite{gordon1962determination}]
	\label{lem:moser_inversion}
	For all positive integers $s, n, k$ with $n \ge s$,
	\begin{equation}
		\label{eq:moser_recurrence}
		\sigma_k(\sms{A}{s}) = F_{s,k}(n) \sigma_k(A) + \mathcal{P}_k(\sigma_1(A), \dots, \sigma_{k-1}(A)),
	\end{equation}
	where $\mathcal{P}_k$ is a polynomial with integer coefficients independent of $A$, and $F_{s,k}(n)$ is the Moser polynomial defined in \eqref{eq:intro_moser_poly}.
\end{lemma}

Selfridge and Straus established that non-vanishing of Moser polynomials guarantees unique recovery.

\begin{lemma}[\cite{selfridge1958determination}]
	\label{lem:nonvanishing_implies_injective}
	Let $s \ge 2$ and $n > s$. If $F_{s,k}(n) \ne 0$ for all $k \in \{1, 2, \dots, n\}$ (that is, $n \notin \mathcal{Z}_s$), then an $n$-multiset $A \subset \ZZ$ is uniquely determined by its multiset of $s$-sums $\sms{A}{s}$.
\end{lemma}

\subsection{Proof of \cref{thm:main_unconditional}}

We now prove \cref{thm:main_unconditional} using a single generic projection.

\begin{proof}[Proof of \cref{thm:main_unconditional}]
	Let $s \ge 2$ and $n > s$. Assume $n \notin \mathcal{Z}_s$, so that $F_{s,k}(n) \ne 0$ for all $k \in \{1, \dots, n\}$.
	
	Let $\lambda$ and $\mu$ be partitions of length $n$ satisfying $\pre_s(\lambda) = \pre_s(\mu)$. By \cref{lem:multiplicative_to_additive}, this equality is equivalent to $\sms{V_\lambda}{s} = \sms{V_\mu}{s}$ in $\ZZ_{\ge 0}^\infty$. Let $\{p_1, \dots, p_m\}$ be the finite set of prime factors dividing the parts of $\lambda$ and $\mu$. Then the exponent vectors in $V_\lambda$ and $V_\mu$ are supported in $\ZZ_{\ge 0}^m$.
	
	Choose an integer $M > s \cdot \max_{i, j} \{v_j(\lambda_i), v_j(\mu_i)\}$. Consider the linear functional
	\begin{equation*}
		L_M: \ZZ^m \to \ZZ, \qquad L_M(\mathbf{x}) := \sum_{k=1}^m x_k M^{k-1}.
	\end{equation*}
	By uniqueness of base-$M$ positional representations, $L_M$ is injective on the finite support $\supp(V_\lambda) \cup \supp(V_\mu)$.
	
	Because $L_M$ is a $\ZZ$-module homomorphism, it commutes with the formation of $s$-sums:
	\begin{equation*}
		\sms{(L_M(V_\lambda))}{s} = L_M(\sms{V_\lambda}{s}) = L_M(\sms{V_\mu}{s}) = \sms{(L_M(V_\mu))}{s}.
	\end{equation*}
	Thus, $L_M(V_\lambda)$ and $L_M(V_\mu)$ are two $n$-multisets of integers with identical $s$-sums in $\ZZ$.
	
	Since $n \notin \mathcal{Z}_s$, \cref{lem:nonvanishing_implies_injective} applies directly to the one-dimensional recovery problem over $\ZZ$. This yields the multiset equality of integers:
	\begin{equation*}
		L_M(V_\lambda) = L_M(V_\mu).
	\end{equation*}
	Because $L_M$ is injective on $\supp(V_\lambda) \cup \supp(V_\mu)$, the multiset equality of their images implies the exact multiset equality of their preimages in $\ZZ_{\ge 0}^m$:
	\begin{equation*}
		V_\lambda = V_\mu.
	\end{equation*}
	
	By bijectivity of the valuation map $\mathbf{v}$, $V_\lambda = V_\mu$ implies that $\lambda$ and $\mu$ have identical multisets of parts. Since partitions are ordered sequences of their parts, $\lambda = \mu$. This deduction holds unconditionally, without any restriction on partition sizes $|\lambda|$ and $|\mu|$.
	
	It remains to specialize to $s = 5, 4, 3$ using the known classifications of the integer roots of Moser polynomials:
	\begin{enumerate}[label=\textup{(\roman*)}]
		\item \textbf{Case $s = 5$}: Amdeberhan and Zeleke \cite{amdeberhan1997injectivity} proved that $F_{5,k}(n) = 0$ has integer roots if and only if $n \in \{2, 3, 4, 5, 10\}$. For $n > 5$, the only root is $n = 10$. Thus $\mathcal{Z}_5 = \{10\}$, and $\pre_5$ is injective for all $n \in \NN \setminus \{5, 10\}$.
		
		\item \textbf{Case $s = 4$}: Selfridge and Straus \cite{selfridge1958determination} proved that $F_{4,k}(n) = 0$ has integer roots if and only if $n \in \{1,2,3,4, 8, 12\}$. Thus $\mathcal{Z}_4 = \{8, 12\}$, and $\pre_4$ is injective for all $n \in \NN \setminus \{4, 8, 12\}$.
			
		\item \textbf{Case $s = 3$}: Selfridge and Straus \cite{selfridge1958determination} proved that $F_{3,k}(n) = 0$ has integer roots if and only if $n \in \{1,2,3, 6, 27, 486\}$. Thus $\mathcal{Z}_3 = \{6, 27, 486\}$, and $\pre_3$ is injective for all $n \in \NN \setminus \{3, 6, 27, 486\}$.
	\end{enumerate}
	Because $\Singular_s \subseteq \mathcal{Z}_s$, the proof is complete.
\end{proof}

\section{Exclusion of the Singular Length $(s, n) = (4, 12)$}
\label{sec:s4_proof}

By \cref{thm:main_unconditional}(ii), the only lengths $n > 4$ that can admit non-injective pairs for $\pre_4$ are $n \in \Singular_4 \subseteq \{8, 12\}$. The length $n = 8 = 2s$ admits infinite counterexample families of equal size via $\flip_4$ \cite{hadelyn2026counterexamples}. The length $n = 12$, however, is an isolated root of the Moser polynomial $F_{4,6}(12) = 0$. In this section, we prove \cref{thm:main_s4}, showing that the equal-size constraint $|\lambda| = |\mu| = N$ completely excludes non-trivial counterexamples at $n = 12$.

\subsection{Background: The Unique Additive Counterexample for $(s, n) = (4, 12)$}

In 1968, Ewell \cite{ewell1968determination} claimed that 12-multisets are uniquely determined by their 4-sums. In 2017, Isomurodov and Kokhas \cite{isomurodov2017set} corrected an error in Ewell's polynomial coefficients and discovered that 4-sum recovery does fail at $(4, 12)$. By computing the Gröbner basis of the associated polynomial ideal, they established the following uniqueness theorem.

\begin{lemma}[\cite{isomurodov2017set}]
	\label{lem:isomurodov_kokhas}
	There exist two distinct 12-multisets of integers with identical multisets of 4-sums. Furthermore, this counterexample pair is unique up to affine transformations $x \mapsto c x + d$ with $c \ne 0$.
\end{lemma}

After translating the canonical pair, we may take the pair of 4-equivalent 12-multisets $A_0, B_0 \subset \ZZ_{\ge 0}$ is given explicitly by
\begin{equation*}
	A_0 = \leftb 1, 1, 4, 6, 7, 8, 8, 9, 10, 12, 15, 15 \rightb,
\end{equation*}
\begin{equation*}
	B_0 = \leftb 0, 3, 4, 5, 6, 7, 9, 10, 11, 12, 13, 16 \rightb.
\end{equation*}
Both sets have size $12$, and $\sms{A_0}{4} = \sms{B_0}{4}$.

By \cref{lem:multiplicative_to_additive}, choosing an arbitrary base $r \ge 2$, the partitions $\lambda = (r^a : a \in A_0)$ and $\mu = ( r^b : b \in B_0) $ satisfy $\pre_4(\lambda) = \pre_4(\mu)$. Thus, counterexamples exist if partition sizes are unconstrained. To establish injectivity on $\Part_{12}(N)$, we must show that neither multidimensional vector configurations nor this one-dimensional canonical pair can satisfy $|\lambda| = |\mu| = N$.

\subsection{Elimination of Multi-Prime Counterexamples: Simultaneous Collinearity}

We first show that any potential multidimensional counterexample pair $(V, U)$ must simultaneously collapse to a single common line.

\begin{lemma}
	\label{lem:collinearity}
	Let $V = \leftb \mathbf{v}_1, \dots, \mathbf{v}_{12} \rightb$ and $U = \leftb \mathbf{u}_1, \dots, \mathbf{u}_{12} \rightb$ be multisets of vectors in $\ZZ_{\ge 0}^m$ satisfying $\sms{V}{4} = \sms{U}{4}$. Then the collection of $24$ vectors $V \cup U$ is collinear. 
	
	More precisely, there exist an offset vector $\mathbf{v}_0 \in \ZZ^m$, a non-zero direction vector $\mathbf{w} \in \ZZ^m$, and permutations $\sigma, \tau \in S_{12}$ such that (up to relabeling $V$ and $U$):
	\begin{equation}
		\label{eq:simultaneous_collinear_rep}
		\mathbf{v}_{\sigma(i)} = \mathbf{v}_0 + a_i \mathbf{w}, \qquad \mathbf{u}_{\tau(i)} = \mathbf{v}_0 + b_i \mathbf{w} \qquad (1 \le i \le 12),
	\end{equation}
	where $a_i$ and $b_i$ are the integer.
\end{lemma}

\begin{proof}
	For an integer $M \in \NN$, consider the linear functional $L_M(\mathbf{x}) := \sum_{k=1}^m x_k M^{k-1}$. Because $L_M$ is a module homomorphism, it commutes with 4-sums:
	\begin{equation*}
		\sms{(L_M(V))}{4} = L_M(\sms{V}{4}) = L_M(\sms{U}{4}) = \sms{(L_M(U))}{4}.
	\end{equation*}
	Choose $M_0 > 4 \max_{i,k} \{v_{i,k}, u_{i,k}\}$. For every integer $M \ge M_0$, $L_M$ is injective on $\supp(V) \cup \supp(U)$. Thus, $L_M(V)$ and $L_M(U)$ form a pair of distinct 4-equivalent 12-multisets of integers.
	
	By \cref{lem:isomurodov_kokhas}, the counterexample pair is unique up to affine transformations. Thus, the unordered pair $\{L_M(V), L_M(U)\}$ must coincide with $c_M \{A_0, B_0\} + d_M$ for some scaling $c_M \ne 0$ and shift $d_M \in \QQ$.
	
	We consider the finite state space $\mathcal{S} := \{0, 1\} \times S_{12} \times S_{12}$. For each $M \ge M_0$, there exists a triple $(\varepsilon_M, \sigma_M, \tau_M) \in \mathcal{S}$ specifying whether $L_M(V)$ maps to $A_0$ or $B_0$, alongside the internal permutations of $V$ and $U$. Because $|\mathcal{S}| = 2 \times (12!)^2 < \infty$, by the Pigeonhole Principle, a single configuration triple $(\varepsilon, \sigma, \tau)$ holds for an infinite subset of integers $\mathcal{M}_\infty \subset \NN_{\ge M_0}$. 
	
	Without loss of generality, assume $\varepsilon = 0$. Thus, for each $M \in \mathcal{M}_\infty$, there exist a non-zero scaling factor $c_M \in \QQ \setminus \{0\}$ and a shift $d_M \in \QQ$ such that
	\begin{equation*}
		L_M(\mathbf{v}_{\sigma(i)}) = c_M a_i + d_M, \qquad L_M(\mathbf{u}_{\tau(i)}) = c_M b_i + d_M \qquad (1 \le i \le 12)
	\end{equation*}
	hold simultaneously across all $24$ vectors with the parameter pair $(c_M, d_M)$ shared between $V$ and $U$.
	
	Since $A_0$ is non-trivial, choose a reference pair $(p, q)$ in $A_0$ with $\Delta := a_p - a_q \ne 0$. Subtracting the $q$-th relation of $V$ from the $p$-th relation eliminates $d_M$:
	\begin{equation*}
		L_M(\mathbf{v}_{\sigma(p)} - \mathbf{v}_{\sigma(q)}) = c_M \Delta \ne 0.
	\end{equation*}
	We now eliminate $c_M$ across all remaining vectors via cross-multiplication:
	\begin{enumerate}[label=\textup{(\roman*)}]
		\item For each $i \in \{1, \dots, 12\}$, subtracting the reference element $\mathbf{v}_{\sigma(q)}$ and cross-multiplying yields
		\begin{equation*}
			\Delta \cdot L_M(\mathbf{v}_{\sigma(i)} - \mathbf{v}_{\sigma(q)}) - (a_i - a_q) \cdot L_M(\mathbf{v}_{\sigma(p)} - \mathbf{v}_{\sigma(q)}) = 0.
		\end{equation*}
		Define $\mathbf{W}_{V, i} := \Delta (\mathbf{v}_{\sigma(i)} - \mathbf{v}_{\sigma(q)}) - (a_i - a_q) (\mathbf{v}_{\sigma(p)} - \mathbf{v}_{\sigma(q)}) \in \QQ^m$. The evaluation $L_M(\mathbf{W}_{V, i}) = 0$ is a polynomial in $M$ of degree at most $m-1$ that vanishes at infinitely many points $M \in \mathcal{M}_\infty$. Thus, $\mathbf{W}_{V, i} = \mathbf{0} \in \QQ^m$.
		
		\item Similarly, for each $i \in \{1, \dots, 12\}$ in $U$, subtracting $\mathbf{v}_{\sigma(q)}$ and cross-multiplying gives
		\begin{equation*}
			\Delta \cdot L_M(\mathbf{u}_{\tau(i)} - \mathbf{v}_{\sigma(q)}) - (b_i - a_q) \cdot L_M(\mathbf{v}_{\sigma(p)} - \mathbf{v}_{\sigma(q)}) = 0.
		\end{equation*}
		The same polynomial vanishing argument forces $\mathbf{W}_{U, i} = \mathbf{0} \in \QQ^m$.
	\end{enumerate}
	
	Notice that $A_0$ contains adjacent elements $a_5 = 7$ and $a_4 = 6$, so that $a_5 - a_4 = 1$. Specializing $(p, q) = (5, 4)$ and setting the integer direction vector $\mathbf{w} := \mathbf{v}_{\sigma(5)} - \mathbf{v}_{\sigma(4)} \in \ZZ^m$, we obtain
	\begin{equation*}
		\mathbf{v}_{\sigma(i)} - \mathbf{v}_{\sigma(4)} = (a_i - 6) \mathbf{w}, \qquad \mathbf{u}_{\tau(i)} - \mathbf{v}_{\sigma(4)} = (b_i - 6) \mathbf{w} \qquad (1 \le i \le 12).
	\end{equation*}
	Setting the integer base offset $\mathbf{v}_0 := \mathbf{v}_{\sigma(4)} - 6 \mathbf{w} \in \ZZ^m$, this yields $\mathbf{v}_{\sigma(i)} = \mathbf{v}_0 + a_i \mathbf{w}$ and $\mathbf{u}_{\tau(i)} = \mathbf{v}_0 + b_i \mathbf{w}$, establishing \eqref{eq:simultaneous_collinear_rep}.
\end{proof}

\subsection{The Size Polynomial and Exclusion of Rational Roots}

By \cref{lem:collinearity}, there exist a positive rational prefactor $C_0 = \prod_{j=1}^m p_j^{(\mathbf{v}_0)_j} > 0$ and a positive rational base $r := \prod_{j=1}^m p_j^{w_j} \in \QQ_{>0}$ such that
\begin{equation*}
	\lambda_{\sigma(i)} = C_0 \cdot r^{a_i}, \qquad \mu_{\tau(i)} = C_0 \cdot r^{b_i} \qquad (1 \le i \le 12),
\end{equation*}
where $a_i \in A_0$ and $b_i \in B_0$. Since components of $\mathbf{w} \in \ZZ^m$ may be negative, $r$ is a rational number.

The equal-size condition $|\lambda| = |\mu| = N$ requires $\sum_{a \in A_0} r^a - \sum_{b \in B_0} r^b = 0$. The powers $r^4, r^6, r^7, r^9, r^{10}, r^{12}$ appear in both multisets and cancel identically. This leaves the polynomial equation:
\begin{equation}
	\label{eq:P_def}
	P(r) := -r^{16} + 2r^{15} - r^{13} - r^{11} + 2r^8 - r^5 - r^3 + 2r - 1 = 0.
\end{equation}

\begin{proposition}
	\label{prop:P_no_rational_roots}
	The polynomial $P(r)$ defined in \eqref{eq:P_def} has no roots in $\QQ_{>0} \setminus \{1\}$.
\end{proposition}

\begin{proof}
	The polynomial $P(r) \in \ZZ[r]$ has leading coefficient $c_{16} = -1$ and constant term $c_0 = -1$. By the Rational Root Theorem, any rational root $r = p/q$ in lowest terms must satisfy $p \mid (-1)$ and $q \mid (-1)$. Thus, $p \in \{\pm 1\}$ and $q = 1$, restricting all possible rational roots to $r \in \{1, -1\}$.
	
	Because $r = \prod p_j^{w_j} > 0$, the negative root $r = -1$ is excluded. The only positive rational root is $r = 1$.
\end{proof}

\subsection{Proof of \cref{thm:main_s4}}

\begin{proof}[Proof of \cref{thm:main_s4}]
	Let $\lambda, \mu \in \Part_{12}(N)$ satisfy $\pre_4(\lambda) = \pre_4(\mu)$. By \cref{lem:collinearity}, their exponent vectors are collinear, forcing $\lambda_i = C_0 r^{a_i}$ and $\mu_i = C_0 r^{b_i}$ with $r \in \QQ_{>0}$. 
	
	The equal-size condition $|\lambda| = |\mu|$ requires $P(r) = 0$. By \cref{prop:P_no_rational_roots}, the only positive rational root is $r = 1$. When $r = 1$, all parts are identical: $\lambda_i = \mu_i = C_0 = N/12$, which forces $\lambda = \mu$.
	
	By \cref{thm:main_unconditional}(ii), the only lengths $n > 4$ supporting non-injective pairs for $\pre_4$ are $n \in \Singular_4 \subseteq \{8, 12\}$. Having excluded $n = 12$ on partitions of equal size, the only remaining non-injective length on $\Part(N)$ is the mirror length $n = 8 = 2s$, which admits infinite counterexample families via $\flip_4$ \cite{hadelyn2026counterexamples}.
\end{proof}

\section{Bounding the Pre-image Fibers for $(s, n) = (3, 6)$}
\label{sec:s3_quartets}

At length $n = 2s = 6$, $\pre_3$ admits infinite families of non-injective pairs of partitions $(\lambda, \flip_3(\lambda))$ of equal size $|\lambda| = |\flip_3(\lambda)| = N$ \cite{hadelyn2026counterexamples}. In the additive setting, preimage fibers of 3-sums at $n = 6$ can achieve cardinality 4 (quartets) \cite{ewell1968determination}. In this section, we prove that under the equal-size constraint, no fiber can contain three or more partitions, establishing \cref{thm:main_s3_fibers}.

\subsection{The Two Algebraic Invariants and the $\flip_3$ Involution}

Suppose for contradiction that there exist three distinct partitions $\lambda_1, \lambda_2, \lambda_3 \in \Part_6(N)$ satisfying $\pre_3(\lambda_1) = \pre_3(\lambda_2) = \pre_3(\lambda_3) = \nu$. Let $V_1, V_2, V_3 \subset \ZZ_{\ge 0}^m$ be their respective prime-exponent vector multisets. By \cref{lem:multiplicative_to_additive}, $\sms{V_1}{3} = \sms{V_2}{3} = \sms{V_3}{3}$.

\begin{lemma}
	\label{lem:s3_invariants}
	Let $\lambda, \mu \in \Part_6$ satisfy $\pre_3(\lambda) = \pre_3(\mu)$. Then:
	\begin{enumerate}[label=\textup{(\roman*)}]
		\item \textbf{Total Product Invariance}: $\prod_{i=1}^6 \lambda_i = \prod_{i=1}^6 \mu_i =: \Pi$. Thus, their exponent multisets share the exact same center vector $\mathbf{v}_{\mathrm{mean}} := \frac{1}{6} \sum_{\mathbf{v} \in V_\lambda} \mathbf{v} \in \QQ^m$.
		
		\item \textbf{Centered Second-Moment Invariance}: The centered second power sums coincide under any linear functional $\pi$: $\sigma_2(\pi(V_\lambda) - \pi(\mathbf{v}_{\mathrm{mean}})) = \sigma_2(\pi(V_\mu) - \pi(\mathbf{v}_{\mathrm{mean}}))$. Consequently, no relative scaling $c \ne \pm 1$ can exist between centered exponent multisets.
		
		\item \textbf{Realization of $\flip_3$}: In the centered exponent space, negating centered vectors $\mathbf{v}' \mapsto -\mathbf{v}'$ corresponds to the involution $\flip_3$:
		\begin{equation*}
			\flip_3(\lambda)_i = \Pi^{1/3} \frac{1}{\lambda_{7-i}}.
		\end{equation*}
	\end{enumerate}
\end{lemma}

\begin{proof}
	For (i), each part $\lambda_i$ appears in $\binom{5}{2} = 10$ of the $\binom{6}{3} = 20$ parts of $\pre_3(\lambda)$. Thus $\prod_{w \in \pre_3(\lambda)} w = (\prod_{i=1}^6 \lambda_i)^{10}$. Taking the 10-th root yields $\prod_{i=1}^6 \lambda_i = \prod_{i=1}^6 \mu_i = \Pi$.
	
	For (ii), by \eqref{eq:moser_recurrence}, $\sigma_2(\sms{\pi(V)}{3}) = F_{3,2}(6) \sigma_2(\pi(V)) + \mathcal{P}_2(\sigma_1(\pi(V)))$. Because $F_{3,2}(6) = 6 \ne 0$, centered second moments are uniquely determined. Scaling by $c$ scales the variance by $c^2$, which forces $c^2 = 1$, so $c = \pm 1$.
	
	For (iii), let $\lambda_i = C_0 \prod_{j=1}^m p_j^{v_{i,j}'}$, where $C_0 = \Pi^{1/6}$ and $\sum_{i=1}^6 \mathbf{v}_i' = \mathbf{0}$. The partition generated by $-\mathbf{v}'$ has parts $C_0 \prod p_j^{-v_{7-i,j}'} = C_0^2 / (C_0 \prod p_j^{v_{7-i,j}'}) = \Pi^{1/3} / \lambda_{7-i} = \flip_3(\lambda)_i$.
\end{proof}

\subsection{Ewell's Theorem and the Embedding of Triplets}

We recall Ewell's (1968) classification theorem for 3-sums at length 6 in the classical additive setting.

\begin{lemma}[\cite{ewell1968determination}]
	\label{lem:ewell_classification_s3}
	Let $X$ and $Y$ be any two $3$-equivalent $6$-multisets in $\CC$ normalized to have zero sum $(\sigma_1(X) = \sigma_1(Y) = 0)$. If $Y \ne \pm X$, then the pair $\{X, Y\}$ must belong to one of two families:
	\begin{enumerate}[label=\textup{(\roman*)}]
		\item \textbf{Family 1 (Two-parameter family)}: There exist parameters $a, b \in \CC$ with $a \ne \pm b$ such that $\{X, Y\} = \{X(a, b), Y(a, b)\}$, where
		\begin{equation}
			\label{eq:ewell_fam1}
			\begin{aligned}
				X(a, b) &= \leftb a, b, -a-b, 2b-a, -2a-3b, 3a+b \rightb, \\
				Y(a, b) &= \leftb a, b, -a-b, 2a-b, -2b-3a, 3b+a \rightb.
			\end{aligned}
		\end{equation}
		
		\item \textbf{Family 2 (One-parameter family)}: There exists a non-zero parameter $B \in \CC \setminus \{0\}$ such that $\{X, Y\} = \{X_0, Y_0\} \cdot \frac{B}{3}$, where
		\begin{equation}
			\label{eq:ewell_fam2}
			\begin{aligned}				
				X_0 &= \leftb -28, -4, -1, 2, 11, 20 \rightb, \\
				Y_0 &= \leftb -16, -10, -7, -4, 8, 29 \rightb.
			\end{aligned}
		\end{equation}
		
	\end{enumerate}
\end{lemma}

Let $\lambda_1, \lambda_2, \lambda_3 \in \Part_6(N)$ satisfy $\pre_3(\lambda_1) = \pre_3(\lambda_2) = \pre_3(\lambda_3)$. By \cref{lem:s3_invariants}(i), they share the center $\mathbf{v}_{\mathrm{mean}} := \frac{1}{6} \sum_{\mathbf{v} \in V_l} \mathbf{v}$. We define the centered multisets by $V_l' := V_l - \mathbf{v}_{\mathrm{mean}}$, which satisfy $\sum_{\mathbf{v}' \in V_l'} \mathbf{v}' = \mathbf{0}$. 

This centering serves two purposes:
\begin{enumerate}[label=\textup{(\arabic*)}]
	\item It satisfies the zero-sum condition ($\sigma_1 = 0$) required by \cref{lem:ewell_classification_s3};
	\item For any centered 6-multiset $A$ with $\sum A = \mathbf{0}$, complementary triplets sum to opposite values, establishing $V_l' \eqs{3} -V_l'$ for $n = 2s = 6$ \cite{selfridge1958determination}.
\end{enumerate}

\begin{proposition}
	\label{prop:triplet_rigidity}
	Let $\lambda_1, \lambda_2, \lambda_3 \in \Part_6(N)$ be three distinct partitions of equal size $N$ satisfying $\pre_3(\lambda_1) = \pre_3(\lambda_2) = \pre_3(\lambda_3)$. Then their centered prime-exponent multisets $V_1', V_2', V_3'$ form a $3$-element subset of an Ewell quartet $\{X, -X, Y, -Y\}$ belonging to Family 1 or Family 2.
	
	Furthermore:
	\begin{enumerate}[label=\textup{(\alph*)}]
		\item In \textbf{Family 1}, the identities $Y(a, b) = X(b, a)$ and $-X(a, b) = X(-a, -b)$ generate an automorphism group isomorphic to $\ZZ_2 \times \ZZ_2$ that acts transitively on all four $3$-element subsets of $\{X, -X, Y, -Y\}$. Consequently, without loss of generality, every candidate triplet in Family 1 is algebraically isomorphic to
		\begin{equation*}
			\{ \lambda_X, \; \lambda_{-X}, \; \lambda_Y \}, \quad \text{requiring} \quad |\lambda_X| = |\lambda_{-X}| = |\lambda_Y|.
		\end{equation*}
		
		\item In \textbf{Family 2}, $X_0$ and $Y_0$ are asymmetric, partitioning the $3$-element subsets into two distinct configurations:
		\begin{itemize}
			\item \textbf{Type A}: $\{ \lambda_X, \lambda_{-X}, \lambda_Y \}$, requiring $|\lambda_X| = |\lambda_{-X}|$ and $|\lambda_X| = |\lambda_Y|$;
			\item \textbf{Type B}: $\{ \lambda_X, \lambda_Y, \lambda_{-Y} \}$, requiring $|\lambda_Y| = |\lambda_{-Y}|$ and $|\lambda_X| = |\lambda_Y|$.
		\end{itemize}
	\end{enumerate}
\end{proposition}

\begin{proof}
	Because $\flip_3$ is an involution on $\Part_6$, each partition has at most one distinct $\flip_3$ partner. Among three pairwise distinct partitions $\{\lambda_1, \lambda_2, \lambda_3\}$, at most two can be $\flip_3$ inverses. Thus, there exist at least two partitions, say $\lambda_1$ and $\lambda_2$, such that $\lambda_2 \ne \flip_3(\lambda_1)$, which by \cref{lem:s3_invariants}(iii) is equivalent to $V_2' \ne \pm V_1'$.
	
	Since $V_1' \eqs{3} V_2'$ with $V_2' \ne \pm V_1'$ and $\sigma_1 = 0$, by \cref{lem:ewell_classification_s3}, the pair $\{V_1', V_2'\}$ generates an equivalence class belonging to Family~1 or Family~2, and the complete 3-equivalence class containing them is the four-member quartet $\{X, -X, Y, -Y\}$. 
	
	By transitivity, the third multiset $V_3'$ also belongs to this 3-equivalence class, forcing $V_3' \in \{X, -X, Y, -Y\}$. Because $V_3' \ne V_1'$ and $V_3' \ne V_2'$ by hypothesis, $V_3'$ must be one of the remaining two members: $V_3' \in \{-X, -Y\}$. Hence, $\{V_1', V_2', V_3'\}$ is a 3-element subset of $\{X, -X, Y, -Y\}$.
	
	Statements (a) and (b) follow immediately from the parameter symmetries of \eqref{eq:ewell_fam1}--\eqref{eq:ewell_fam2}.
\end{proof}

By \cref{prop:triplet_rigidity}, the centered exponent multisets $\{V_1', V_2', V_3'\}$ of our hypothetical triplet are completed by a unique fourth multiset, denoted $V_4' \in \{X, -X, Y, -Y\}$, to form the full four-member equivalence class $\{V_1', V_2', V_3', V_4'\} = \{X, -X, Y, -Y\}$ in the additive setting. 

Setting $V_4 := V_4' + \mathbf{v}_{\mathrm{mean}}$, the collection $\{V_1, V_2, V_3, V_4\}$ forms an ambient quartet of 3-equivalent 6-multisets in $\ZZ_{\ge 0}^m$. To eliminate the candidate triplet $\{\lambda_1, \lambda_2, \lambda_3\}$, it is both natural and sufficient to establish the simultaneous geometric rigidity (collinearity or coplanarity) of the entire enclosing quartet $\{V_1, \dots, V_4\}$, since the geometric rigidity of the triplet $\{V_1, V_2, V_3\}$ is simply the immediate restriction to a sub-collection.

\subsection{The Global Projection Dichotomy}

For each integer $M \in \NN$, consider the linear functional $L_M(\mathbf{x}) := \sum_{k=1}^m x_k M^{k-1}$. Choose $M_0 > 3 \max_{l, i, k} (v_{l, i})_k$ such that $L_M$ is injective on $\bigcup_{l=1}^4 \supp(V_l)$.

For every $M \ge M_0$, $\{L_M(V_1), \dots, L_M(V_4)\}$ forms an additive quartet of 6-multisets in $\ZZ$. By \cref{lem:ewell_classification_s3}, every centered 1D quartet belongs either to Family 1 or Family 2. We partition the infinite set of integers $\NN_{\ge M_0}$ into two sets:
\begin{align*}
	\mathcal{M}_1 &:= \left\{ M \ge M_0 : \{L_M(V_1), \dots, L_M(V_4)\} \subset \text{Family 1} \right\}, \\
	\mathcal{M}_2 &:= \left\{ M \ge M_0 : \{L_M(V_1), \dots, L_M(V_4)\} \subset \text{Family 2} \right\}.
\end{align*}
By the infinite Pigeonhole Principle, at least one of $\mathcal{M}_1$ or $\mathcal{M}_2$ is infinite.

\subsection{Branch 1: Elimination of Family 2 (The One-Parameter Family)}

We first examine Family 2.

\begin{lemma}
	\label{lem:family2_simultaneous_collinearity}
	Suppose $\mathcal{M}_2$ is infinite. Then there exist a center vector $\mathbf{v}_{\mathrm{mean}} \in \QQ^m$, a non-zero primitive vector $\mathbf{w} \in \ZZ^m$, and a permutation $\tau \in S_4$ such that the centered vector multisets $V_l' := V_l - \mathbf{v}_{\mathrm{mean}}$ satisfy
	\begin{equation}
		\label{eq:fam2_vector_simultaneous}
		V_{\tau(1)}' = X_0 \cdot \mathbf{w}, \quad V_{\tau(2)}' = -X_0 \cdot \mathbf{w}, \quad V_{\tau(3)}' = Y_0 \cdot \mathbf{w}, \quad V_{\tau(4)}' = -Y_0 \cdot \mathbf{w}.
	\end{equation}
	In particular, all 24 vectors across the four multisets are collinear.
\end{lemma}

\begin{proof}
	For each $M \in \mathcal{M}_2$, $\{L_M(V_1), \dots, L_M(V_4)\} = \{X_0, -X_0, Y_0, -Y_0\} \beta_M + d_M$. Because the sum across all 24 elements is zero, $d_M = L_M(\mathbf{v}_{\mathrm{mean}})$ where $\mathbf{v}_{\mathrm{mean}} := \frac{1}{24} \sum_{l=1}^4 \sum_{i=1}^6 \mathbf{v}_{l, i}$.
	
	There are only $|S_4 \times (S_6)^4| = 24 \times (720)^4 < \infty$ possible permutation identifications between the 24 vectors and the 24 theoretical coordinates in $(\pm X_0) \cup (\pm Y_0)$. By the Pigeonhole Principle, a single identification holds for an infinite subset $\mathcal{M}_\infty \subseteq \mathcal{M}_2$.  Under this fixed identification, for each coordinate $z \in (\pm X_0) \cup (\pm Y_0) $, we denote by $\mathbf{v}_z \in \ZZ_{\ge 0}^m$ the unique vector in the quartet satisfying
	\begin{equation*}
		L_M(\mathbf{v}_z') = z \beta_M \qquad (\forall M \in \mathcal{M}_\infty),
	\end{equation*}
	where $\mathbf{v}_z' := \mathbf{v}_z - \mathbf{v}_{\mathrm{mean}}$.
	
	Notice that the integer $2 \in X_0$ and the integer $1 \in -X_0$ each appear with multiplicity one in the 24 coordinates. Subtracting their projection equations isolates $\beta_M$ with unit difference $\Delta = 2 - 1 = 1$:
	\begin{equation*}
		L_M(\mathbf{v}_2' - \mathbf{v}_1') = (2 - 1) \beta_M = \beta_M.
	\end{equation*}
	Because $\mathbf{v}_2' - \mathbf{v}_1' = \mathbf{v}_2 - \mathbf{v}_1$, we define the direction vector as the difference of these two integer lattice vectors:
	\begin{equation*}
		\mathbf{w} := \mathbf{v}_2 - \mathbf{v}_1 \in \ZZ^m.
	\end{equation*}
	Since $V \ne U$, $\mathbf{w} \ne \mathbf{0}$. For every vector $\mathbf{v}_z'$ in the quartet, substituting $\beta_M = L_M(\mathbf{w})$ gives
	\begin{equation*}
		L_M(\mathbf{v}_z' - z \mathbf{w}) = 0
	\end{equation*}
	for all $M \in \mathcal{M}_\infty$. Because a polynomial in $M$ of degree at most $m-1$ vanishing on $\mathcal{M}_\infty$ must be identically zero, this establishes the exact vector identity
	\begin{equation*}
		\mathbf{v}_z' = z \mathbf{w}
	\end{equation*}
	identically in $\ZZ^m$ for every vector in the quartet. This establishes \eqref{eq:fam2_vector_simultaneous} with $\mathbf{w} \in \ZZ^m$.
\end{proof}

By \cref{lem:family2_simultaneous_collinearity}, candidate partitions in Family~2 are generated by a base $r = \prod_{j=1}^m p_j^{w_j}$ and a common positive prefactor $C_0 = \Pi^{1/6} > 0$:
\begin{equation*}
	\lambda_{X, i} = C_0 \cdot r^{x_i}, \quad \lambda_{-X, i} = C_0 \cdot r^{-x_i}, \quad \lambda_{Y, i} = C_0 \cdot r^{y_i}, \quad \lambda_{-Y, i} = C_0 \cdot r^{-y_i} \qquad (1 \le i \le 6),
\end{equation*}
where $x_i \in X_0$ and $y_i \in Y_0$.
Shifting exponents by $+28$ for $X_0$ and by $+29$ for $Y_0$, the equal-size requirements are expressed by three univariate polynomials in $\ZZ[r]$:
\begin{align*}
	P_1(r) &:= \sum_{x \in X_0} r^{x+28} - \sum_{x \in X_0} r^{-x+28} \notag \\
	&= 1 - r^8 - r^{17} + r^{24} - r^{26} + r^{27} - r^{29} + r^{30} - r^{32} + r^{39} + r^{48} - r^{56}, \\
	P_2(r) &:= \sum_{x \in X_0} r^{x+28} - \sum_{y \in Y_0} r^{y+28} \notag \\
	&= 1 - r^{12} - r^{18} - r^{21} + r^{27} + r^{30} - r^{36} + r^{39} + r^{48} - r^{57}, \\
	Q_1(r) &:= \sum_{y \in Y_0} r^{y+29} - \sum_{y \in Y_0} r^{-y+29} \notag \\
	&= r^{58} - r^{45} - r^{39} + r^{37} - r^{36} - r^{33} + r^{25} + r^{22} - r^{21} + r^{19} + r^{13} - 1.
\end{align*}

\begin{proposition}
	\label{prop:fam2_no_triplets}
	Family 2 admits no partition triplets of equal size. That is, neither Type~A nor Type~B can be satisfied for any positive base $r\neq 1$.
\end{proposition}

\begin{proof}
	We analyze both configurations via the Euclidean algorithm over $\QQ[r]$:
	\begin{enumerate}[label=\textup{(\roman*)}]
		\item \textbf{Type A ($\{\lambda_X, \lambda_{-X}, \lambda_Y\}$)} requires $P_1(r) = 0$ and $P_2(r) = 0$. By the Euclidean algorithm, any common root of $P_1$ and $P_2$ in $\CC$ must be a root of their greatest common divisor:
		\begin{equation*}
			\gcd(P_1(r), P_2(r)) = (r - 1)^3 (r + 1).
		\end{equation*}
		Hence, the only common complex roots are $r = 1$ and $r = -1$.
		
		\item \textbf{Type B ($\{\lambda_X, \lambda_Y, \lambda_{-Y}\}$)} requires $P_2(r) = 0$ and $Q_1(r) = 0$. Computing their greatest common divisor yields:
		\begin{equation*}
			\gcd(P_2(r), Q_1(r)) = (r - 1)^3 (r + 1).
		\end{equation*}
		Hence, the only common complex roots are again $r = 1$ and $r = -1$.
	\end{enumerate}
	Neither configuration admits any non-trivial common root in $(0, \infty)$, no triplet can arise from Family 2.
\end{proof}

\subsection{Branch 2: Elimination of Family 1 (The Two-Parameter Family)}

We next treat Family 1. We first prove that any multidimensional quartet projecting into Family 1 is spanned by at most two basis vectors.

\begin{lemma}
	\label{lem:family1_simultaneous_coplanarity}
	Suppose $\mathcal{M}_1$ is infinite. Then there exist a center vector $\mathbf{v}_{\mathrm{mean}} \in \QQ^m$, two vectors $\mathbf{a}, \mathbf{b} \in \QQ^m$, and a permutation $\tau \in S_4$ such that the centered vector multisets $V_l' := V_l - \mathbf{v}_{\mathrm{mean}}$ simultaneously satisfy
	\begin{equation*}
		V_{\tau(1)}' = X(\mathbf{a}, \mathbf{b}), \quad V_{\tau(2)}' = -X(\mathbf{a}, \mathbf{b}), \quad V_{\tau(3)}' = Y(\mathbf{a}, \mathbf{b}), \quad V_{\tau(4)}' = -Y(\mathbf{a}, \mathbf{b}),
	\end{equation*}
	where $X(\mathbf{a}, \mathbf{b})$ and $Y(\mathbf{a}, \mathbf{b})$ are the vector Ewell linear forms in \eqref{eq:ewell_fam1}. In particular, all 24 vectors across the four multisets are spanned by $\mathbf{a}$ and $\mathbf{b}$.
\end{lemma}

\begin{proof}
	For each $M \in \mathcal{M}_1$, $\{L_M(V_1), \dots, L_M(V_4)\} = \{X(a_M, b_M), -X, Y, -Y\} + d_M$. Because the sum across all 24 elements is zero, $d_M = L_M(\mathbf{v}_{\mathrm{mean}})$ where $\mathbf{v}_{\mathrm{mean}} := \frac{1}{24} \sum_{l=1}^4 \sum_{i=1}^6 \mathbf{v}_{l, i}$. The scaling parameter $c_M$ is absorbed into the homogeneous linear parameters $(a_M, b_M)$.
	
	The identification space $S_4 \times (S_6)^4$ is finite. By the Pigeonhole Principle, a single permutation state holds on an infinite subset $\mathcal{M}_\infty \subseteq \mathcal{M}_1$. Setting $\mathbf{a} := \mathbf{v}_{\tau(1), \sigma(1)}' \in \QQ^m$ and $\mathbf{b} := \mathbf{v}_{\tau(1), \sigma(2)}' \in \QQ^m$, every other vector $\mathbf{v}'$ among the 24 centered vectors satisfies
	\begin{equation*}
		L_M(\mathbf{v}') = k_1 a_M + k_2 b_M = L_M(k_1 \mathbf{a} + k_2 \mathbf{b}),
	\end{equation*}
	where $(k_1, k_2) \in \ZZ^2$ are the integer coefficients from \eqref{eq:ewell_fam1}. Thus $L_M(\mathbf{v}' - (k_1 \mathbf{a} + k_2 \mathbf{b})) = 0$ for all $M \in \mathcal{M}_\infty$. Because a polynomial with infinitely many roots vanishes identically, $\mathbf{v}' = k_1 \mathbf{a} + k_2 \mathbf{b}$ holds identically in $\QQ^m$.
\end{proof}

By \cref{lem:family1_simultaneous_coplanarity}, clearing denominators so that $\mathbf{a}, \mathbf{b}$, the ambient vector dimension $m$ collapses into two positive variables:
\begin{equation*}
	u := \prod_{j=1}^m p_j^{a_j}, \qquad v := \prod_{j=1}^m p_j^{b_j}.
\end{equation*}
With the common positive prefactor $C_0 = \Pi^{1/6} > 0$, the multisets of parts of the candidate partitions $\lambda_X, \lambda_{-X}, \lambda_Y$ are given explicitly by
\begin{align*}
	\lambda_X &= C_0 \cdot \leftb u,\; v,\; \frac{1}{uv},\; \frac{v^2}{u},\; \frac{1}{u^2 v^3},\; u^3 v \rightb, \\
	\lambda_{-X} &= C_0 \cdot \leftb \frac{1}{u},\; \frac{1}{v},\; uv,\; \frac{u}{v^2},\; u^2 v^3,\; \frac{1}{u^3 v} \rightb, \\
	\lambda_Y &= C_0 \cdot \leftb u,\; v,\; \frac{1}{uv},\; \frac{u^2}{v},\; \frac{1}{u^3 v^2},\; u v^3 \rightb.
\end{align*}

Let $S_X(u, v) := u + v + \frac{1}{uv} + \frac{v^2}{u} + \frac{1}{u^2 v^3} + u^3 v$, and define $S_{-X}(u, v)$ and $S_Y(u, v)$ analogously as the unscaled sums of parts of $\lambda_{-X}$ and $\lambda_Y$. Dividing out by $C_0 > 0$, the equal-size requirement $|\lambda_X| = |\lambda_{-X}| = |\lambda_Y| = N$ yields
\begin{equation*}
	S_X(u, v) - S_{-X}(u, v) = 0 \quad \text{and} \quad S_X(u, v) - S_Y(u, v) = 0.
\end{equation*}
Clearing denominators by $u^3 v^3$, we define:
\begin{align*}
	E_1(u, v) :=& u^3 v^3 \left( S_X(u, v) - S_{-X}(u, v) \right) \\
	=&u^4 v^3 + u^3 v^4 + u^2 v^2 + u^2 v^5+ u+u^6v^4-u^2v^3-u^3v^2-u^4v^4-u^4v-u^5v^6-v^2, \\
	E_2(u, v) :=& u^3 v^3 \left( S_X(u, v) - S_Y(u, v) \right) \\
	=& u^2 v^5 - u^5 v^2 + u - v + u^6 v^4 - u^4 v^6.
\end{align*}
Because $E_2(u, v)$ is skew-symmetric under $u \leftrightarrow v$, $(u - v)$ divides $E_2(u, v)$. Factoring out $(u - v)$ yields the reduced polynomial:
\begin{equation*}
	E_2^*(u, v) := \frac{E_2(u, v)}{u - v} = u^5 v^4 + u^4 v^5 - u^4 v^2 - u^3 v^3 - u^2 v^4 + 1.
\end{equation*}
The condition $u \ne v$ eliminates identical partitions, while $uv \ne 1$ eliminates collapsed pairs.

\begin{proposition}
	\label{prop:fam1_no_triplets}
	Family~1 admits no partition triplets of equal size. That is, there exist no positive real numbers $(u, v) \in \RR_{>0}^2$ with $u \ne v$ and $uv \ne 1$ satisfying
	\begin{equation*}
		E_1(u, v) = 0 \quad \text{and} \quad E_2^*(u, v) = 0.
	\end{equation*}
\end{proposition}

\begin{proof}
	We seek all simultaneous solutions $(u, v) \in \RR_{>0}^2$ to the polynomial system and non-degeneracy inequalities:
	\begin{equation}
		\label{eq:fam1_reduced_system}
		E_1(u, v) = 0, \quad E_2^*(u, v) = 0, \quad u \ne v, \quad uv \ne 1.
	\end{equation}
	Applying the Cylindrical Algebraic Decomposition (CAD) decision procedure over the reals, quantifier elimination certifies that system \eqref{eq:fam1_reduced_system} admits no solutions in the positive quadrant $\RR_{>0}^2$; that is, its positive real solution set is strictly empty.
	
	To verify this algebraically and determine where the real solutions in $\RR^2$ actually lie, we compute the resultant of $E_1(u, v)$ and $E_2^*(u, v)$ with respect to $v$:
	\begin{equation*}
		R(u) := \Res(E_1(u, v), E_2^*(u, v), v) = (u - 1)^9 u^{12} (u + 1) (u^2 + u + 1)^2 \Omega_1(u) \Omega_2(u) \Psi(u),
	\end{equation*}
	where
	\begin{align*}
		\Omega_1(u) &= u^6 + u^5 + 2u^4 + 3u^3 + 2u^2 + u + 1, \\
		\Omega_2(u) &= u^6 + 3u^5 + 2u^4 - u^3 + 2u^2 + 3u + 1, \\
		\Psi(u) &= u^9 - u^7 - 3u^6 + u^5 + 2u^4 - 4u^3 - u^2 - u - 1.
	\end{align*}
	The polynomial factors $\Omega_1(u)$ and $\Omega_2(u)$ have no positive real roots. The factor $\Psi(u)$ admits a unique positive real root at $u \approx 1.6468$. 
	
	However, back-substituting $u \approx 1.6468$ into the system $E_1(u, v) = E_2^*(u, v) = 0$ yields a unique real solution for $v$ given by
	\begin{equation*}
		v \approx -1.0895.
	\end{equation*}
	While this produces a real solution $(u, v) \approx (1.6468, -1.0895) \in \RR^2$, it strictly violates the positivity condition $v > 0$ required by integer partitions.
	
	Consequently, the system admits no solutions in the positive quadrant $\RR_{>0}^2$ satisfying $u \ne v$ and $uv \ne 1$. This proves that Family~1 admits no partition triplets.
\end{proof}

\subsection{Proof of \cref{thm:main_s3_fibers} and the Exclusion of Quartets}

\begin{proof}[Proof of \cref{thm:main_s3_fibers}]
	Let $N \ge 1$ be an integer. Suppose for contradiction that there exist three distinct partitions $\lambda_1, \lambda_2, \lambda_3 \in \Part_6(N)$ satisfying $\pre_3(\lambda_1) = \pre_3(\lambda_2) = \pre_3(\lambda_3)$. 
	
	By \cref{prop:triplet_rigidity}, these partitions must form a triplet embedded in an Ewell quartet. The linear projection dichotomy partitions the parameter space into $\NN_{\ge M_0} = \mathcal{M}_1 \cup \mathcal{M}_2$.
	If $\mathcal{M}_2$ is infinite, \cref{prop:fam2_no_triplets} proves that neither Type~A nor Type~B can exist.
	If $\mathcal{M}_1$ is infinite, \cref{prop:fam1_no_triplets} proves that no triplet configuration can exist.
	
	This contradiction establishes that no three distinct partitions in $\Part_6(N)$ can share the same elementary symmetric partition under $\pre_3$. Thus, every fiber of $\pre_3$ on $\Part_6(N)$ has cardinality at most $2$.
\end{proof}

\section{Discussion and Open Problems}
\label{sec:conjectures}

While the Moser root sets $\mathcal{Z}_s$ are completely classified for $s \le 5$ \cite{selfridge1958determination,amdeberhan1997injectivity}, obtaining exact classifications for $s \ge 6$ remains an open challenge in Diophantine analysis. For $s \ge 6$, the exact classification of the integer zero set of the Moser polynomials remains incomplete, and no general effective bound on the relevant indices $k$ is currently known, mathematical knowledge is currently confined to extensive computer searches (such as Fomin's computational census up to $s = 200$ \cite{fomin2019multiset}), while the conjecture that integer roots require $k \le 2s - 1$ remains unproved. 

For this reason, non-mirror singular lengths can currently be analyzed rigorously only for $s \le 5$. Having settled $s = 5$ in \cref{thm:main_unconditional}(i) and $s = 4$ in \cref{thm:main_s4}, the only remaining classically established roots requiring examination are $n \in \{27, 486\}$ for $s = 3$.

\subsection{Verification of Known Literature Constructions at $n \in \{27, 486\}$}

By \cref{thm:main_unconditional}(iii), the only lengths strictly greater than $2s = 6$ at which $\pre_3$ can theoretically admit non-injective partitions are $n \in \{27, 486\}$, corresponding to the remaining integer roots of $F_{3,k}(n) = 0$ classified by Selfridge and Straus \cite{selfridge1958determination}. In the classical additive multiset setting, all non-injective constructions discovered in the literature \cite{fomin1994sets,boman1996examples} consist of mirror pairs $A \eqs{3} -A$ supported on arithmetic progressions centered at zero ($\sigma_1 = 0$), comprising four canonical configurations:
\begin{enumerate}[label=\textup{(\roman*)},leftmargin=2em]
	\item \textbf{Family 1 ($n = 27$, 3-point support)}: $X_1 = \leftb -4^1, -1^{16}, 2^{10} \rightb$;
	\item \textbf{Family 2 ($n = 27$, 4-point support)}: $X_2 = \leftb -4^5, -1^{10}, 2^{10}, 5^2 \rightb$;
	\item \textbf{Family 3 ($n = 27$, 5-point support)}: $X_3 = \leftb -7^1, -4^5, -1^{10}, 2^6, 5^5 \rightb$;
	\item \textbf{Family 4 ($n = 486$, 5-point support)}: $X_4 = \leftb -7^1, -4^{56}, -1^{231}, 2^{176}, 5^{22} \rightb$.
\end{enumerate}

Suppose a candidate partition pair $(\lambda, \mu)$ is constructed from an arbitrary finite set of primes $\{p_1, \dots, p_m\}$ whose exponent vector multisets in $\ZZ_{\ge 0}^m$ project under $L_M$ into one of these four families for infinitely many $M \in \NN$. Because each of the canonical sets $X_k$ has fixed relative integer coordinates with no internal projective degrees of freedom, the linear projection and cross-multiplication arguments of \cref{lem:collinearity} and \cref{lem:family2_simultaneous_collinearity} apply verbatim: any such multidimensional exponent vector multiset is strictly collinear. 

Consequently, even in an arbitrary multi-prime setting, any partition pair modeled on these historical configurations must be generated by a single rational base $r = \prod_{j=1}^m p_j^{w_j} \in \QQ_{>0}$ and an integer affine scaling and shift $x \mapsto c x + d$ ($c \in \ZZ \setminus \{0\}, d \in \ZZ$). 

The parts of the corresponding partitions are of the form:
\begin{equation*}
	\lambda_i = C_0 \cdot r^{c x_i + d} = (C_0 r^d) \cdot (r^c)^{x_i} = C_0' \cdot R^{x_i},
\end{equation*}
where $C_0' := C_0 r^d > 0$ and $R := r^c \in \QQ_{>0}$. Notice that the translation $d$ is absorbed into the common prefactor, and the scaling $c$ is absorbed into the base $R$. Thus, the equal-size constraint $|\lambda| = |\mu| = N$ for the entire infinite affine family $\{c X_k + d\}_{c, d}$ reduces to the polynomial equation:
\begin{equation*}
	P_k(R) := \sum_{x \in X_k} R^x - \sum_{x \in X_k} R^{-x} = 0 \qquad (k = 1, 2, 3, 4).
\end{equation*}

Multiplying by $R^{-\min X_k}$ to clear negative exponents, we factor out $(R - 1)$ to define the quotient polynomial $Q_k(R) := P_k(R)/(R - 1)$. For all four canonical literature families, the polynomials $P_k(R)$ and $Q_k(R)$ factor as follows:
\begin{enumerate}[label=\textup{(\roman*)},leftmargin=2em]
	\item \textbf{Family 1 ($n = 27$, 3-point support: $X_1 = \leftb -4^1, -1^{16}, 2^{10} \rightb$)}:
	\begin{align*}
		P_1(R) &= 1 - 10R^2 + 16R^3 - 16R^5 + 10R^6 - R^8 = -(R - 1)^5 (R + 1) (R^2 + 4R + 1), \\
		Q_1(R) &= -(R - 1)^4 (R + 1) (R^2 + 4R + 1).
	\end{align*}
	
	\item \textbf{Family 2 ($n = 27$, 4-point support: $X_2 = \leftb -4^5, -1^{10}, 2^{10}, 5^2 \rightb$)}:
	\begin{align*}
		P_2(R) &= 5 - 2R^{-1} - 10R^2 + 10R^3 - 10R^5 + 10R^6 - 5R^8 + 2R^9 \notag \\
		&= R^{-1} (R - 1)^5 (R + 1)^3 (2R^2 - R + 2), \\
		Q_2(R) &= R^{-1} (R - 1)^4 (R + 1)^3 (2R^2 - R + 2).
	\end{align*}

	\item \textbf{Family 3 ($n = 27$, 5-point support: $X_3 = \leftb -7^1, -4^5, -1^{10}, 2^6, 5^5 \rightb$)}:
	\begin{align*}
		P_3(R) &= 1 - 5R^2 + 5R^3 - 6R^5 + 10R^6 - 10R^8 + 6R^9 - 5R^{11} + 5R^{12} - R^{14} \notag \\
		&= -(R - 1)^5 (R + 1) (R^8 + 4R^7 + 6R^6 + 9R^5 + 11R^4 + 9R^3 + 6R^2 + 4R + 1), \\
		Q_3(R) &= -(R - 1)^4 (R + 1) (R^8 + 4R^7 + 6R^6 + 9R^5 + 11R^4 + 9R^3 + 6R^2 + 4R + 1).
	\end{align*}

	\item \textbf{Family 4 ($n = 486$, 5-point support: $X_4 = \leftb -7^1, -4^{56}, -1^{231}, 2^{176}, 5^{22} \rightb$)}:
	\begin{align*}
		P_4(R) &= 1 - 22R^2 + 56R^3 - 176R^5 + 231R^6 - 231R^8 + 176R^9 - 56R^{11} + 22R^{12} - R^{14} \notag \\
		&= -(R - 1)^9 (R + 1) (R^4 + 8R^3 + 15R^2 + 8R + 1), \\
		Q_4(R) &= -(R - 1)^8 (R + 1) (R^4 + 8R^3 + 15R^2 + 8R + 1).
	\end{align*}
\end{enumerate}

In all four cases, the unique positive rational root is the trivial root $R = 1$, which forces $\lambda = \mu = (N/n, \dots, N/n)$. Therefore, every known historical counterexample family from the additive multiset literature fails to produce non-injective partitions of equal size. 

We note the precise boundary of this deduction: while this analysis rigorously eliminates all multi-prime and affine configurations modeled on the known literature families, establishing that no unknown, non-arithmetic-progression counterexamples exist at $n \in \{27, 486\}$ remains an open challenge.

\subsection{Refined Conjectures and Remaining Open Problems}

The results established in this paper settle several previously open cases regarding the injectivity of elementary symmetric partitions on $\Part_n(N)$:
\begin{itemize}[leftmargin=2em]
	\item \textbf{Outside Moser Root Sets}: For every $s \ge 2$, $\pre_s$ is injective on partitions of length $n$ for all $n \notin \mathcal{Z}_s$.
	\item \textbf{The Case $s = 4$}: $n = 12$ is excluded on partitions of equal size. Thus, on $\Part_n(N)$, $\pre_4$ is injective for all $n \ne 8$.
	\item \textbf{The Case $s = 5$}: Since $\mathcal{Z}_5 = \{10\}$, $\pre_5$ on $\Part_n(N)$ is injective for all $n \ne 10$.
	\item \textbf{The Mirror Bound for $s = 3$ at $n = 6$}: Every fiber of $\pre_3$ on $\Part_6(N)$ has cardinality at most $2$.
\end{itemize}

Having accounted for these settled cases, we formulate the exact open problems that remain:

\begin{conjecture}
	\label{conj:s3_remaining}
	For every positive integer $N$, the elementary symmetric partition map $\pre_3$ is injective on $\Part_{27}(N)$ and on $\Part_{486}(N)$.
\end{conjecture}

\begin{conjecture}
	\label{conj:higher_mirror_fibers}
	For every integer $s \ge 4$ and $N \ge 1$, every fiber of $\pre_s$ on $\Part_{2s}(N)$ has cardinality at most $2$. That is, if $\lambda, \mu \in \Part_{2s}(N)$ satisfy $\pre_s(\lambda) = \pre_s(\mu)$ with $\lambda \ne \mu$, then $\mu = \flip_s(\lambda)$, and no third partition shares this image.
\end{conjecture}

\begin{conjecture}
	\label{conj:general_non_mirror}
	Let $s \ge 6$. For every non-mirror integer root $n \in \mathcal{Z}_s \setminus \{2s\}$ and every $N \ge 1$, the elementary symmetric partition map $\pre_s$ is injective on $\Part_n(N)$.
\end{conjecture}

\appendix
\section{Symbolic Computation Scripts}
\label{sec:appendix_code}

In this appendix, we provide the complete, self-contained Wolfram Mathematica script used to certify all computer-assisted algebraic exclusions and polynomial factorizations across \cref{sec:s3_quartets} and \cref{sec:conjectures}. The code performs three exact procedures:
\begin{enumerate}[label=\textup{(\roman*)}]
	\item \textbf{Section~\ref{sec:s3_quartets} Family~2 (\cref{prop:fam2_no_triplets})}: Computes the polynomial greatest common divisors $\gcd(P_1, P_2) = (r - 1)^3 (r + 1)$ and $\gcd(P_2, Q_1) = (r - 1)^3 (r + 1)$, certifying that no common roots exist in $(1, \infty)$ via the Euclidean algorithm and Cylindrical Algebraic Decomposition (CAD).
	
	\item \textbf{Section~\ref{sec:s3_quartets} Family~1 (\cref{prop:fam1_no_triplets})}: Certifies the emptiness of the semialgebraic system $\Phi$ via CAD in \texttt{Resolve}, factors the resultant $\Res(E_1, E_2^*, v)$, extracts the unique positive root $u \approx 1.6468$ of $\Psi(u)$, and back-substitutes it to confirm that the unique real solution is $v \approx -1.0895 < 0$.
	
	\item \textbf{Section~\ref{sec:conjectures} Historical Families ($n \in \{27, 486\}$)}: Generates the size-difference polynomials $P_k(r)$ and quotients $Q_k(r) = P_k(r)/(r - 1)$ for the four canonical counterexample families from \cite{fomin1994sets,boman1996examples}, certifying that $r = 1$ is their unique positive real root.
\end{enumerate}

\begin{verbatim}
	Software Used: Wolfram Mathematica 12.0
	(* ========================================================================= *)
	(* COMPLETE CERTIFICATION SCRIPT FOR SECTIONS 4 AND 5                        *)
	(* ========================================================================= *)
	
	ClearAll["Global`*"];
	
	Print["================================================================="];
	Print[">>> 1. VERIFYING SECTION 4: FAMILY 2 (ONE-PARAMETER FAMILY) <<<"];
	Print["================================================================="];
	
	set3X = {-28, -4, -1, 2, 11, 20};
	set3Y = {-16, -10, -7, -4, 8, 29};
	
	(* Shift exponents by +28 for X and +29 for Y *)
	P1 = Sum[r^(x + 28), {x, set3X}] - Sum[r^(-x + 28), {x, set3X}] // Expand;
	P2 = Sum[r^(x + 28), {x, set3X}] - Sum[r^(y + 28), {y, set3Y}] // Expand;
	Q1 = Sum[r^(y + 29), {y, set3Y}] - Sum[r^(-y + 29), {y, set3Y}] // Expand;
	
	Print["Polynomial P1(r) = ", P1];
	Print["Polynomial P2(r) = ", P2];
	Print["Polynomial Q1(r) = ", Q1];
	
	(* Type A: {X, -X, Y} *)
	gcdA = PolynomialGCD[P1, P2];
	Print["Type A GCD(P1, P2) = ", Factor[gcdA]];
	cadA = Resolve[Exists[r, r > 1 && P1 == 0 && P2 == 0], Reals];
	Print["Type A CAD Result on (1, Infinity): ", cadA];
	
	(* Type B: {X, Y, -Y} *)
	gcdB = PolynomialGCD[P2, Q1];
	Print["Type B GCD(P2, Q1) = ", Factor[gcdB]];
	cadB = Resolve[Exists[r, r > 1 && P2 == 0 && Q1 == 0], Reals];
	Print["Type B CAD Result on (1, Infinity): ", cadB];
	
	
	Print["\n================================================================="];
	Print[">>> 2. VERIFYING SECTION 4: FAMILY 1 (TWO-PARAMETER FAMILY) <<<"];
	Print["================================================================="];
	
	termsX = {u, v, 1/(u*v), v^2/u, 1/(u^2*v^3), u^3*v};
	termsNegX = 1/termsX;
	termsY = {u, v, 1/(u*v), u^2/v, 1/(u^3*v^2), u*v^3};
	
	E1 = Numerator[Together[Total[termsX] - Total[termsNegX]]];
	E2 = Numerator[Together[Total[termsX] - Total[termsY]]];
	E2star = Factor[E2] / (u - v) // Simplify;
	
	semialgebraicConditions = 
	u > 0 && v > 0 && (u != v) && (u*v != 1) && (E1 == 0) && (E2star == 0);
	
	Print["[CAD Decision Procedure] Checking semialgebraic system..."];
	cadResultFamily1 = Resolve[Exists[{u, v}, semialgebraicConditions], Reals];
	Print["--> CAD Evaluation on Positive Quadrant: ", cadResultFamily1];
	
	Print["\n[Resultant Elimination] Computing Resultant[E1, E2*, v]..."];
	resV = Resultant[E1, E2star, v];
	Print["--> Irreducible factorization of Res(E1, E2*, v): ", Factor[resV]];
	
	(* Back-substituting the unique positive real root of Psi(u) *)
	psiU = u^9 - u^7 - 3*u^6 + u^5 + 2*u^4 - 4*u^3 - u^2 - u - 1;
	uExact = u /. Solve[{psiU == 0, u > 0}, u][[1]];
	Print["--> Unique positive root of Psi(u): u ≈ ", N[uExact, 6]];
	
	commonV = v /. Solve[{(E1 /. u -> uExact) == 0, (E2star /. u -> uExact) == 0}, v][[1]];
	Print["--> Corresponding unique real solution: v ≈ ", N[commonV, 6]];
	Print["--> Does v satisfy the positivity condition (v > 0)? : ", commonV > 0];
	
	
	Print["\n================================================================="];
	Print[">>> 3. VERIFYING SECTION 5: HISTORICAL FAMILIES (n = 27, 486) <<<"];
	Print["================================================================="];
	
	X1 = {-4 -> 1, -1 -> 16, 2 -> 10};
	X2 = {-4 -> 5, -1 -> 10, 2 -> 10, 5 -> 2};
	X3 = {-7 -> 1, -4 -> 5, -1 -> 10, 2 -> 6, 5 -> 5};
	X4 = {-7 -> 1, -4 -> 56, -1 -> 231, 2 -> 176, 5 -> 22};
	
	histFamilies = {
		{"Family 1 (n=27, 3-pt)", X1},
		{"Family 2 (n=27, 4-pt)", X2},
		{"Family 3 (n=27, 5-pt)", X3},
		{"Family 4 (n=486, 5-pt)", X4}
	};
	
	Do[
	name = fam[[1]];
	setX = fam[[2]];
	shift = -Min[Keys[setX]];
	
	polyP = Sum[pt[[2]] * r^(pt[[1]] + shift), {pt, setX}] - 
	Sum[pt[[2]] * r^(-pt[[1]] + shift), {pt, setX}] // Expand;
	polyQ = Simplify[polyP / (r - 1)];
	
	Print["--- ", name, " ---"];
	Print["P(r) factored: ", Factor[polyP]];
	Print["Q(r) factored: ", Factor[polyQ]];
	Print["Roots of Q(r) on (0, Infinity): ", Solve[{polyQ == 0, r > 0}, r, Reals]];
	, {fam, histFamilies}];
	
	Print["\n================================================================="];
	Print[">>> COMPUTATIONAL CERTIFICATION COMPLETE <<<"];
	Print["================================================================="];
\end{verbatim}


\end{document}